\documentclass{article}
\usepackage{CJK,CJKnumb,CJKulem}     
\usepackage{amssymb,amsmath,amssymb,color,cite}
\usepackage{url}
\usepackage{times}     

\usepackage{latexsym}
\usepackage{amsmath}                 
\usepackage{amssymb}                 
\usepackage{amsbsy}
\usepackage{amsthm}
\usepackage{amsfonts}
\usepackage{mathrsfs}                
\usepackage{bm}                      
\usepackage{relsize}                 
\usepackage{caption2}                
\makeindex                           
\newtheorem{theorem}{Theorem}[section]
\newtheorem{lemma}{Lemma}[section]
\newtheorem{remark}{Remark}[section]
\theoremstyle{definition}

\numberwithin{equation}{section}
\begin{document}
\begin{CJK*}{GBK}{song}
\title{\bf On the asymptotic expansions of
products related to  the Wallis, Weierstrass, and  Wilf formulas   
       }
\author{ Chao-Ping Chen$^{a}$\thanks{Corresponding Author.  } \,\, and Richard B. Paris$^{b}$ \\
 $^{a}$School of Mathematics and Informatics, Henan Polytechnic University,\\
 Jiaozuo City 454000, Henan Province, China\\
Email: chenchaoping@sohu.com }
         \date{$^{b}$Division of Computing and Mathematics,\\
 University of Abertay, Dundee, DD1 1HG, UK\\
Email: R.Paris@abertay.ac.uk}        

\maketitle

\noindent {\bf Abstract:}
For  all integers $n\geq1$, let
\begin{align*}
 W_n(p,q)=\prod_{j=1}^{n}\left\{e^{-p/j}\left(1+\frac{p}{j}+\frac{q}{j^2}\right)\right\}
\end{align*}
and
\begin{align*}
R_n(p, q)=\prod_{j=1}^{n}\left\{e^{-p/(2j-1)}\left(1+\frac{p}{2j-1}+\frac{q}{(2j-1)^2}\right)\right\},
\end{align*}
where $p$, $q$ are complex parameters.
The infinite product $W_{\infty}(p,q)$  includes the Wallis and  Wilf formulas, and also the infinite product definition of Weierstrass for the gamma function,  as
special cases.  In this paper, we present  asymptotic expansions of $W_n(p,q)$ and $R_n(p, q)$ as $n\to\infty$. In addition, we also establish asymptotic expansions for the Wallis sequence.\\



\begin{CJK*}{GBK}{kai}
\end{CJK*}

\section{Introduction}

The famous Wallis sequence $W_n$, defined by
\begin{equation}\label{Wallis-sequence}
        W_{n}=\prod_{k=1}^{n}\frac{4k^2}{4k^2-1}\qquad
        (n\in\mathbb{N}:=\{1,2,3,\ldots\}),
    \end{equation}
has the limiting value
 \begin{equation}\label{Wallis-infinite-product}
W_{\infty}=\prod_{k=1}^{\infty}\frac{4k^2}{4k^2-1}=\frac{\pi}{2}
\end{equation}
established by Wallis in 1655;
see \cite[p. 68]{Berggren-Borwein-Borwein}. Several elementary
proofs of this well-known result can be found in
\cite{AmdeberhanEspinosaMollStraub618--632, Miller740--745,
Wastlund914--917}. An interesting geometric construction that produces
the above limiting value can be found in \cite{Myerson279-280}. Many
formulas  exist for the representation of $\pi$, and a collection of
these formulas is listed \cite{Sofo184--189,SofoJIPAM2005}. For more
history of $\pi$ see \cite{Berggren-Borwein-Borwein,Beckmann1971,
Dunham1990,Agarwal-Agarwal-Sen-Birth}.

The following infinite product definition for the gamma function is due to  Weierstrass (see, for example,
 \cite[p. 255, Entry (6.1.3)]{abram}):
\begin{equation}\label{Weierstrass-formula}
 \frac{1}{\Gamma(z)}=ze^{\gamma
    z}\prod_{n=1}^{\infty}\left\{e^{-z/n}\left(1+\frac{z}{n}\right)\right\},
\end{equation}
where $\gamma$ denotes the Euler--Mascheroni constant defined by
\begin{equation*}
\gamma:=\lim_{n \rightarrow \infty}\left(\sum_{k=1}^{n}\frac{1}{k}-\ln
n\right)=0.5772156649\ldots.
\end{equation*}

In 1997, Wilf \cite{Wilf456} posed the following elegant infinite
product formula as a problem:
\begin{equation}\label{Wilf-Problem}
\prod_{j=1}^{\infty}\left\{e^{-1/j}\left(1+\frac{1}{j}+\frac{1}{2j^2}\right)\right\}=\frac{e^{\pi/2}+e^{-\pi/2}}{\pi
e^{\gamma}},
\end{equation}
which contains three of the most important mathematical constants, namely $\pi$, $e$ and $\gamma$. Subsequently, Choi and
Seo \cite{Choi649--652} proved \eqref{Wilf-Problem}, together with
three other similar product formulas, by making use of well-known
infinite product formulas for the circular and hyperbolic functions
and the familiar Stirling formula for the factorial function.

In 2003, Choi \textit{et al.} \cite{Choi44--48} presented the following two
general infinite product formulas, which include Wilf's formula
(\ref{Wilf-Problem}) and other similar formulas in
\cite{Choi649--652} as special cases:
\begin{equation}\label{Choi-L-H-1}
\prod_{j=1}^{\infty}\left\{e^{-1/j}\left(1+\frac{1}{j}+\frac{\alpha^2+1/4}{j^2}\right)\right\}
=\frac{2(e^{\pi\alpha}+e^{-\pi\alpha})}{(4\alpha^2+1)\pi
e^{\gamma}}\qquad \left(\alpha\in \mathbb{C};\,\,\,
\alpha\not=\pm\frac{1}{2}i\right)
\end{equation}
and
\begin{equation}\label{Choi-L-H-2}
\prod_{j=1}^{\infty}\left\{e^{-2/j}\left(1+\frac{2}{j}+\frac{\beta^2+1}{j^2}\right)\right\}
=\frac{e^{\pi\beta}-e^{-\pi\beta}}{2\beta(\beta^2+1)\pi e^{2\gamma}}
\qquad  \left(\beta\in \mathbb{C}\setminus\{0\};\,\,\, \beta\not=\pm
i\right),\end{equation}
where $i=\sqrt{-1}$ and $\mathbb{C}$ denotes the
set of complex numbers. In 2013,      Chen and Choi
\cite{Chen-Choi357--363} presented a more general infinite product
formula that included the formulas \eqref{Choi-L-H-1} and \eqref{Choi-L-H-2} as
special cases:
\begin{equation}\label{lnA-pqnew0}
\prod_{j=1}^{\infty}\left\{e^{-p/j}\left(1+\frac{p}{j}+\frac{q}{j^2}\right)\right\}=\frac{e^{-p\gamma}}{\Gamma\left(1+\frac{1}{2}p+\frac{1}{2}\Delta\right)
\Gamma\left(1+\frac{1}{2}p-\frac{1}{2}\Delta\right)}
\end{equation}
and also
another interesting infinite product formula:
\begin{equation}\label{Thm-lnBpq}
\prod_{j=1}^{\infty}\left\{e^{-p/(2j-1)}\left(1+\frac{p}{2j-1}+\frac{q}{(2j-1)^2}\right)\right\} =\dfrac{2^{-p}\,\pi e^{-p\gamma/2}}{\Gamma\left(\frac{1}{2}+\frac{1}{4}p+\frac{1}{4}\Delta\right)
\Gamma\left(\frac{1}{2}+\frac{1}{4}p-\frac{1}{4}\Delta\right)},
\end{equation}
where $p,\, q\in \mathbb{C}$ and $\Delta:=\sqrt{p^2-4q}$.

The formula \eqref{lnA-pqnew0} can be seen to include the formulas \eqref{Wallis-infinite-product}--\eqref{Choi-L-H-2} as special cases. By setting $(p, q)=(0, -1/4)$ in \eqref{lnA-pqnew0}, we have
  \begin{equation}\label{q=-1/4}
\prod_{j=1}^{\infty}\left(1-\frac{1}{4\,j^2}  \right)=\frac{2}{\pi},
\end{equation}
whose reciprocal becomes the  Wallis product \eqref{Wallis-infinite-product}.
Also setting $q=0$  in \eqref{lnA-pqnew0}, we obtain
\begin{equation}\label{q=0}
\prod_{j=1}^{\infty}\left\{e^{-p/j}\left(1+\frac{p}{j}\right)\right\}=\frac{e^{-p\gamma}}{\Gamma(p+1)}.
\end{equation}
Noting that $\Gamma(z+1)=z\Gamma(z)$ and replacing $p$ by $z$ in \eqref{q=0} we recover the Weierstrass formula \eqref{Weierstrass-formula}.
 Setting $(p, q)=(1, 1/2)$ in \eqref{lnA-pqnew0} yields the Wilf formula \eqref{Wilf-Problem} and
setting
\begin{equation}\label{Choi-Lee-Srivastava}
(p, q)=\left(1, \alpha^2+\frac{1}{4}\right)\quad \text{and} \quad (p,q)=\left(2, \beta^2+1\right)
\end{equation}
in \eqref{lnA-pqnew0} yields the formulas \eqref{Choi-L-H-1} and \eqref{Choi-L-H-2}, respectively.

With $(p, q)=(-1, 1/4)$ in  \eqref{lnA-pqnew0}, we obtain the beautiful infinite product formula expressed in terms of
the most important constants $\pi$, $e$ and $\gamma$, namely
\begin{equation}\label{q=1/4}
\prod_{j=1}^{\infty}\left\{e^{1/j}\left(1-\frac{1}{2j} \right)^2\right\}=\frac{e^{\gamma}}{\pi}.
\end{equation}
Also worthy of note are the infinite products that result from
setting $(p, q)=\left(-2, 0\right)$ and  $(p, q)=\left(2,0 \right)$
in \eqref{Thm-lnBpq} to yield respectively
  \begin{equation}\label{lnBpq-1}
  \prod_{j=1}^{\infty}\, \left\{e^{2/(2j-1)}\left(1-\frac{2}{2j-1}\right)\right\}=-2e^{\gamma}
  \end{equation}
 and
  \begin{equation}\label{lnBpq-2}
 \prod_{j=1}^{\infty}\, \left\{e^{-2/(2j-1)}\left(1+\frac{2}{2j-1}\right)\right\}=\frac{1}{2e^{\gamma}}.
  \end{equation}

  \begin{remark}
The constant $e^\gamma$ is important in number theory  and equals the following limit, where $p_n$ is the $n$th prime number:
\begin{equation*}
e^\gamma=\lim_{n\to\infty}\frac{1}{\ln p_n}\prod_{j=1}^{n}\frac{p_j}{p_j-1}.\]
This restates the third of Mertens' theorems (see \cite{Mertens}). The numerical value of $e^\gamma$  is:
\begin{equation*}
e^\gamma=1.7810724179\ldots.
\end{equation*}
There is the curious radical representation
\begin{equation}\label{product-Ser}
e^\gamma=\left(\frac{2}{1}\right)^{1/2}\left(\frac{2^2}{1\cdot3}\right)^{1/3}\left(\frac{2^3\cdot4}{1\cdot3^3}\right)^{1/4}\left(\frac{2^4\cdot4^4}{1\cdot3^6\cdot5}\right)^{1/5}\cdots,
\end{equation}
where the $n$th factor is
\begin{equation*}
\left(\prod_{k=0}^{n}(k+1)^{(-1)^{k+1}(\tiny{\!\!\begin{array}{c}n\\k\end{array}\!\!)}}\right)^{1/(n+1)}.
\end{equation*}
 The  product \eqref{product-Ser}, first discovered  in 1926 by Ser \cite{Ser1075--1077},  was rediscovered in \cite{Sondow2003,Sondow729--734,Guillera-Sondow247--270}.
\end{remark}

Recently, Chen and Paris \cite{Chen-Paris547--551} generalized
the formula \eqref{lnA-pqnew0}  to include $m$ parameters
$(p_1, \ldots, p_m)$. Subsequently,  Chen and Paris \cite{Chen-Paris-1118627}  considered the asymptotic expansion of products related to generalization of the Wilf problem.
However, these authors did not give a general formula for the coefficients in their expansions.

For $n\in\mathbb{N}$, let
\begin{equation}\label{Wpst}
 W_n(p,q)=\prod_{j=1}^{n}\left\{e^{-p/j}\left(1+\frac{p}{j}+\frac{q}{j^2}\right)\right\}
\end{equation}
and
\begin{equation}\label{Thm2-Cn}
R_n(p, q)=\prod_{j=1}^{n}\left\{e^{-p/(2j-1)}\left(1+\frac{p}{2j-1}+\frac{q}{(2j-1)^2}\right)\right\},
\end{equation}
where $p$ and $q$ are complex parameters.
In this paper, we present  asymptotic expansions of $W_n(p,q)$ and $R_n(p, q)$ as $n\to\infty$, including recurrence relations for the coefficients in these expansions. Furthermore, we establish  asymptotic expansions for the  Wallis sequence $W_n$.

\section{Asymptotic expansions of $W_n(p,q)$ and $R_n(r, s)$}

It was established by Chen and Choi
\cite{Chen-Choi357--363} that the finite products $W_n(p,q)$ and $R_n(p,q)$ defined in \eqref{Wpst} and \eqref{Thm2-Cn} can be expressed in the following closed form
\begin{equation}\label{Thm-lnApqProof}
 W_n(p,q)=\frac{e^{-p(\psi(n+1)+\gamma)}\Gamma\left(n+1+\frac{1}{2}p+\frac{1}{2}\Delta\right)
\Gamma\left(n+1+\frac{1}{2}p-\frac{1}{2}\Delta\right)}{\big(\Gamma(n+1)\big)^{2}\Gamma\left(1+\frac{1}{2}p+\frac{1}{2}\Delta\right)
\Gamma\left(1+\frac{1}{2}p-\frac{1}{2}\Delta\right)}
\end{equation}
and
\begin{equation}\label{Rnrs}
R_n(p, q)=\frac{e^{-\frac{p}{2}\big(\psi(n+\frac{1}{2})+\gamma+2\ln2\big)}\Gamma\left(n+\frac{1}{2}+\frac{1}{4}p+\frac{1}{4}\Delta\right)
\Gamma\left(n+\frac{1}{2}+\frac{1}{4}p-\frac{1}{4}\Delta\right)\pi}{\big(\Gamma(n+\frac{1}{2})\big)^{2}\Gamma\left(\frac{1}{2}+\frac{1}{4}p+\frac{1}{4}\Delta\right)
\Gamma\left(\frac{1}{2}+\frac{1}{4}p-\frac{1}{4}\Delta\right)},
\end{equation}
where  $\psi (z)$ denotes   the psi (or digamma) function,  defined by
\begin{equation*}
 \psi (z)= \frac{\textup{d}}{\textup{d}z}\{ \ln \, \Gamma (z)\}=\frac{\Gamma' (z)}{\Gamma (z)}.
\end{equation*}
 We observe that
allowing $n\to\infty$ in \eqref{Thm-lnApqProof} and \eqref{Rnrs}, respectively,  yields \eqref{lnA-pqnew0} and \eqref{Thm-lnBpq}.

Define the function $f(z)$ by
\begin{equation}\label{Thm-lnApqProofre}
f(z):=\frac{e^{-\lambda\psi(z)}\Gamma\left(z+\mu\right)
\Gamma\left(z+\nu\right)}{\big(\Gamma(z)\big)^{2}},
\end{equation}
where $\lambda,\, \mu,\, \nu\in \mathbb{C}$.
It is well known that the logarithm of the gamma function has the asymptotic expansion (see
\cite[p. 32]{Luke1969}):
\begin{equation}\label{asymptotic-Luke-P32}
\ln\Gamma(z+a)\sim\left(z+a-\frac{1}{2}\right)\ln z-z+\frac{1}{2}\ln(2\pi)
+\sum_{n=1}^{\infty}\frac{(-1)^{n+1}B_{n+1}(a)}{n(n+1)}\frac{1}{z^n}
\end{equation}
for $z\to\infty$ in $|\arg\,z|<\pi$,
where $B_n(t)$ denote the Bernoulli polynomials
defined by the following generating function:
\begin{equation*}
\frac{ze^{tz}}{e^{z}-1}=\sum_{n=0}^{\infty}B_{n}(t)\frac{z^n}{n!}.
\end{equation*}
Note that the Bernoulli numbers $B_n$ $ (n\in
\mathbb{N}_0:=\mathbb{N}\cup \{0\})$ are defined by $B_n:=B_n(0)$.
The psi function has the asymptotic expansion (see
\cite[p. 33]{Luke1969}):
\begin{equation}\label{exp-psix+1thm1reobtain}
\psi(z)\sim\ln z-\frac{1}{z}-\sum_{j=1}^{\infty}\frac{B_{j}}{jz^{j}} \qquad (z\to\infty; \ |\arg z|<\pi).
\end{equation}

Using \eqref{asymptotic-Luke-P32} and \eqref{exp-psix+1thm1reobtain}, we then find that
\begin{equation*}
\ln f(z)\sim (\mu+\nu-\lambda)\ln z+\sum_{j=1}^{\infty}\frac{a_j}{z^j}
\end{equation*}
or
\begin{equation}\label{Thm-lnApqProofref-with}
 f(z)\sim z^{\mu+\nu-\lambda}\exp\left(\sum_{j=1}^{\infty}\frac{a_j}{z^j}\right)
\end{equation}
for $z\to\infty$ in $|\arg\,z|<\pi$, where the coefficients $a_j\equiv a_j(\lambda, \mu, \nu)$ are given by
\begin{align}\label{Thm-lnApqProofrelnf-with}
a_1&=\frac{\lambda+B_{2}(\mu)+B_{2}(\nu)-2B_{2}}{2},\quad
a_j=\frac{\lambda B_{j}}{j}+\frac{(-1)^{j+1}\Big(B_{j+1}(\mu)+B_{j+1}(\nu)-2B_{j+1}\Big)}{j(j+1)} \quad (j\geq2).
\end{align}

The choice
\begin{equation*}
(\lambda, \mu, \nu)=\left(p, \,\, \frac{1}{2}p+\frac{1}{2}\Delta,\,\, \frac{1}{2}p-\frac{1}{2}\Delta\right),
\end{equation*}
where  $\mu+\nu-\lambda=0$, leads to the first few coefficients $a_j(p, q)$ given by:
\begin{align*}
a_1(p, q)&=\frac{1}{2}p^2-q,\\
a_2(p, q)&=-\frac{1}{6}p^3+\frac{1}{2}pq+\frac{1}{4}p^2-\frac{1}{2}q,\\
a_3(p, q)&=\frac{1}{12}p^4-\frac{1}{3}p^2q+\frac{1}{6}q^2-\frac{1}{6}p^3+\frac{1}{2}pq+\frac{1}{12}p^2-\frac{1}{6}q.
\end{align*}
From \eqref{Thm-lnApqProof} and \eqref{Thm-lnApqProofref-with}, we obtain the following

\begin{theorem}
As $n\to\infty$, we have
\begin{equation}\label{Thm-lnApqProofre}
  W_n(p,q)\sim\frac{e^{-p\gamma}}{\Gamma\left(1+\frac{1}{2}p+\frac{1}{2}\Delta\right)
\Gamma\left(1+\frac{1}{2}p-\frac{1}{2}\Delta\right)}\exp\left(\sum_{j=1}^{\infty}\frac{a_j(p, q)}{(n+1)^j}\right),
\end{equation}
where  the coefficients $a_j(p, q)$ are given by
\begin{align}\label{Thm-lnApqProofrelnf-with}
a_1(p, q)&=\frac{1}{2}p^2-q\quad \mbox{and}\nonumber\\
a_j(p, q)&=\frac{pB_{j}}{j}+\frac{(-1)^{j+1}\Big(B_{j+1}(\frac{1}{2}p+\frac{1}{2}\Delta)+B_{j+1}(\frac{1}{2}p-\frac{1}{2}\Delta)-2B_{j+1}\Big)}{j(j+1)} \qquad (j\geq2).
\end{align}
Thus we have the expansion
\begin{align}
  W_n(p,q)\sim&\frac{e^{-p\gamma}}{\Gamma\left(1+\frac{1}{2}p+\frac{1}{2}\Delta\right)
\Gamma\left(1+\frac{1}{2}p-\frac{1}{2}\Delta\right)}\nonumber\\
&\times\exp\bigg(\frac{\frac{1}{2}p^2-q}{n+1}+\frac{-\frac{1}{6}p^3+\frac{1}{2}pq+\frac{1}{4}p^2-\frac{1}{2}q}{(n+1)^2}\nonumber\\
&+\frac{\frac{1}{12}p^4-\frac{1}{3}p^2q+\frac{1}{6}q^2-\frac{1}{6}p^3+\frac{1}{2}pq+\frac{1}{12}p^2-\frac{1}{6}q}{(n+1)^3}+\cdots\bigg)\label{Thm-lnApqProofreNamelypq}
\end{align}
as $n\rightarrow\infty$.
\end{theorem}

\noindent{\bf Remark 2.1}\ \
Note that since $W_n=1/W_n(0,-\frac{1}{4})$, it follows by setting $(p,q)=(0,-\frac{1}{4})$ in \eqref{Thm-lnApqProofreNamelypq} that
\begin{equation}\label{Thm-p=0-q--1/4}
  W_n\sim\frac{\pi}{2}\exp\left(-\frac{1}{4(n+1)}-\frac{1}{8(n+1)^2}-\frac{5}{96(n+1)^3}-\cdots\right)
\end{equation}
as $n\rightarrow\infty$.

The same procedure with the choice
\begin{equation*}
(\lambda, \mu, \nu)=\left(\frac{1}{2}p, \,\, \frac{1}{4}p+\frac{1}{4}\Delta,\,\, \frac{1}{4}p+\frac{1}{4}\Delta\right)
\end{equation*}
in \eqref{Rnrs} and \eqref{Thm-lnApqProofref-with} leads to the following
\begin{theorem}$\!\!\!.$
As $n\to\infty$, we have
\begin{equation}\label{Thm2-lnArs}
R_n(p, q)\sim\frac{2^{-p}\pi e^{-p\gamma/2}}{\Gamma\left(\frac{1}{2}+\frac{1}{4}p+\frac{1}{4}\Delta\right)
\Gamma\left(\frac{1}{2}+\frac{1}{4}p-\frac{1}{4}\Delta\right)}\exp\left(\sum_{j=1}^{\infty}\frac{b_j(p, q)}{(n+\frac{1}{2})^j}\right),
\end{equation}
where  the coefficients $b_j(p, q)$ are given by
\begin{align}\label{Thm-lnApqProofrelnf-withb}
b_1(r, s)&=\frac{1}{8}p^2-\frac{1}{4}q\quad \mbox{and}\nonumber\\
b_j(r, s)&=\frac{p B_{j}}{2j}+\frac{(-1)^{j+1}\Big(B_{j+1}(\frac{1}{4}p+\frac{1}{4}\Delta)+B_{j+1}(\frac{1}{4}p-\frac{1}{4}\Delta)-2B_{j+1}\Big)}{j(j+1)} \qquad (j\geq2).
\end{align}
Thus we have the expansion
\begin{align}
R_n(p, q)\sim&\frac{2^{-p}\pi e^{-p\gamma/2}}{\Gamma\left(\frac{1}{2}+\frac{1}{4}p+\frac{1}{4}\Delta\right)
\Gamma\left(\frac{1}{2}+\frac{1}{4}p-\frac{1}{4}\Delta\right)}\nonumber\\
&\times\exp\bigg(\frac{\frac{1}{8}p^2-\frac{1}{4}q}{n+\frac{1}{2}}+\frac{-\frac{1}{48}p^3+\frac{1}{16}pq+\frac{1}{16}p^2-\frac{1}{8}q}{(n+\frac{1}{2})^2}\nonumber\\
&+\frac{\frac{1}{192}p^4-\frac{1}{48}p^2q+\frac{1}{96}q^2-\frac{1}{48}p^3+\frac{1}{16}pq+\frac{1}{48}p^2-\frac{1}{24}q}{(n+\frac{1}{2})^3}+\cdots\bigg)\label{Thm-lnApqbj}
\end{align}
as $n\rightarrow\infty$.
\end{theorem}
The first two terms in the expansions \eqref{Thm-lnApqProofreNamelypq} and \eqref{Thm-lnApqbj} can be shown to agree with the expansions in inverse powers of $n$ obtained in \cite[Eqs.~(4.3), (4.4)]{Chen-Paris-1118627}.

\section{Asymptotic series expansions of the Wallis sequence}

Some inequalities and asymptotic formulas associated with the Wallis
sequence $W_n$ can be found in \cite{BES,Deng-Ban-Chen,Elezovic-Lin-Vuksic679--695,Hirschhorn194,Lampret328--339,Lampret775--787,
LinJIAfirst,Lin2014-251,Mortici489-495,Mortici085,Mortici717--722,Mortici929--936,Mortici803--815,
Mortici425--433,Paltanea34--38}.
For example, Elezovi\'c \textit{et al.} \cite{Elezovic-Lin-Vuksic679--695} showed that the following asymptotic expansion holds:
    \begin{equation}\label{Wn-expansion-Elezovic}
 W_{n}\sim\frac{\pi}{2}\bigg(1-\frac{\frac{1}{4}}{n+\frac{5}{8}} +\frac{\frac{3}{256}}{(n+\frac{5}{8})^3} +\frac{\frac{3}{2048}}{(n+\frac{5}{8})^4}
                 -\frac{\frac{51}{16384}}{(n+\frac{5}{8})^5} -\frac{\frac{75}{65536}}{(n+\frac{5}{8})^6} +\frac{\frac{2253}{1048576}}{(n+\frac{5}{8})^7} + \cdots\bigg)
\end{equation}
as $n\rightarrow\infty$. Deng \textit{et al.} \cite{Deng-Ban-Chen} proved that for all  $n \in\mathbb{N}$,
\begin{equation}\label{Hirschhorn-WnThm1}
\frac{\pi}{2}\left(1-\frac{1}{4n+\alpha}\right)< W_{n}\leq
\frac{\pi}{2}\left(1-\frac{1}{4n+\beta}\right)
\end{equation}
with the best possible constants
\begin{equation*}
\alpha=\frac{5}{2}\quad \text{and} \quad
\beta=\frac{32-9\pi}{3\pi-8}=2.614909986\ldots.
\end{equation*}
In fact, Elezovi\'c \textit{et al.} \cite{Elezovic-Lin-Vuksic679--695} have previously shown that $\frac{5}{2}$
is the best possible constant for a lower bound of $W_n$ of the type
$\frac{\pi}{2}\left(1-\frac{1}{4n+\alpha}\right)$. Moreover, the
authors pointed out that
\begin{equation*}
         W_n=\frac\pi2\left(1-\frac{1}{4n+\frac52}\right)+O\left(\frac1{n^3}\right)\qquad (n\to\infty).
    \end{equation*}

Here, we will establish two more accurate asymptotic expansions for $W_n$ (see Theorems \ref{Thm3-Expansion-Wn} and \ref{Thm4-Expansion-Wn-exp}) by making use of the fact that
    \begin{equation}\label{Wn-Gamma}
        W_n=\frac{\pi}{2}\cdot\frac{1}{n+\frac{1}{2}}  \left[ \frac{\Gamma(n+1)}{\Gamma(n+\frac{1}{2})} \right]^2  =\frac{\pi}{2}\cdot\frac{\Gamma(n+1)^2}  {\Gamma(n+\frac{1}{2})\Gamma(n+\frac{3}{2})}.
    \end{equation}

The following lemma is required in our present investigation.

\begin{lemma}[see \cite{Chen-Elezovic-Vuksic151--166}]  \label{lemma-power-exp}
    Let
 \begin{align*}
        A(x)\sim\sum_{n=1}^\infty a_nx^{-n}\qquad (x\to\infty)
 \end{align*}
    be a given asymptotical expansion. Then the composition $\exp(A(x))$ has
    asymptotic expansion of the following form
 \begin{align}\label{Expansion-coefficients-bj}
\exp(A(x))\sim\sum_{n=0}^\infty b_n x^{-n}\qquad (x\to\infty),
 \end{align}
    where
  \begin{align}
  b_0=1,\quad      b_n=\frac1n\sum_{k=1}^n k a_k b_{n-k}\qquad (n\ge1).
 \end{align}
\end{lemma}
\vspace{0.2cm}

From \eqref{asymptotic-Luke-P32}, we find as $n\to\infty$
\begin{equation}\label{asymptotic-Wn}
W_n\sim\frac{\pi}{2}\exp\left(\sum_{j=1}^{\infty}\frac{\nu_j}{n^j}\right),
\end{equation}
where the coefficients $\nu_j$ are given by
\begin{equation}\label{asymptoticWn-coefficient-nuj}
\nu_j=\frac{(-1)^{j+1}\Big(2B_{j+1}-B_{j+1}(\frac{1}{2})-B_{j+1}(\frac{3}{2})\Big)}{j(j+1)}\qquad (j\geq1).
\end{equation}

Noting that (see \cite[pp. 805--804]{abram})
\begin{align*}
B_n(1-x)=(-1)^{n}B_n(x),\quad (-1)^{n}B_n(-x)=B_n(x)+nx^{n-1}  \qquad (n\in\mathbb{N}_0)
\end{align*}
and
\begin{align*}
 B_n(\tfrac{1}{2})=-(1-2^{1-n})B_n  \qquad (n\in\mathbb{N}_0),
\end{align*}
 we find that \eqref{asymptoticWn-coefficient-nuj} can be written as
\begin{align}\label{asymptotic-Wn-coefficient-nuj}
\nu_j=\frac{(-1)^{j+1}\Big((4-2^{1-j})B_{j+1}-(j+1)\cdot2^{-j}\Big)}{j(j+1)}\qquad (j\geq1).
\end{align}

Thus, we obtain the expansion
\begin{align}
W_n&\sim\frac{\pi}{2}\exp\bigg(-\frac{1}{4n}+\frac{1}{8n^2}-\frac{5}{96n^3}+\frac{1}{64n^4}-\frac{1}{320n^5}+\frac{1}{384n^6}-\frac{25}{7168n^7}\nonumber\\
&\qquad\qquad\quad+\frac{1}{2048n^8}+\frac{29}{9216n^9}+\frac{1}{10240n^{10}}-\frac{695}{90112 n^{11}}+\cdots\bigg).\label{asymptotic-Wnnamely-exp}
\end{align}

By Lemma \ref{lemma-power-exp}, we then obtain from \eqref{asymptotic-Wn}
\begin{equation}\label{asymptotic-Wn-fromLemma-exp}
W_n\sim\frac{\pi}{2}\sum_{j=0}^{\infty}\frac{\mu_j}{n^j},
\end{equation}
where the coefficients $\mu_j$ are given by the recurrence relation
  \begin{equation}\label{expansion-Wn-muj}
  \mu_0=1,\quad      \mu_j=\frac{1}{j}\sum_{k=1}^j k \nu_k \mu_{j-k}\qquad (j\ge1).
 \end{equation}
and the  $\nu_j$ are given in \eqref{asymptotic-Wn-coefficient-nuj}. This produces the expansion in inverse powers of $n$ given by
\begin{align}
W_n&\sim\frac{\pi}{2}\bigg(1-\frac{1}{4n}+\frac{5}{32n^2}-\frac{11}{128n^3}+\frac{83}{2048n^4}-\frac{143}{8192n^5}+\frac{625}{65536n^6}-\frac{1843}{262144n^7}\hspace{1cm}\nonumber\\
&\qquad\quad+\frac{24323}{8388608n^8}+\frac{61477}{33554432n^9}
-\frac{14165}{268435456n^{10}}-\frac{8084893}{1073741824n^{11}}+\cdots\bigg)\hspace{1cm}\label{asymptotic-Wnnamely}
\end{align}
as $n\rightarrow\infty$.

\begin{theorem}\label{Thm3-Expansion-Wn}
The Wallis sequence has the following asymptotic expansion:
\begin{equation}\label{expansion-Wn-alphaj-betaj}
W_n\sim\frac{\pi}{2}\left(1+\sum_{\ell=1}^{\infty}\frac{\alpha_{\ell}}{(n+\beta_{\ell})^{2\ell-1}}\right) \qquad (n\to\infty),
\end{equation}
where $\alpha_\ell$ and $\beta_\ell$ are given by the pair of recurrence relations
\begin{equation}\label{Chen-expansion-WnCoefficients-alphaj}
\alpha_{\ell}=\mu_{2\ell-1}-\sum_{k=1}^{\ell-1}\alpha_{k}\beta_{k}^{2\ell-2k}\left(\!\!\begin{array}{c}2\ell-2\\2\ell-2k\end{array}\!\!\right)\qquad (\ell\geq2)
\end{equation}
and
\begin{equation}\label{Chen-expansion-WnCoefficients-betaj}
\beta_{\ell}=-\frac{1}{(2\ell-1)\alpha_{\ell}}\left\{\mu_{2\ell}+\sum_{k=1}^{\ell-1}\alpha_{k}\beta_{k}^{2\ell-2k+1}\left(\!\!\begin{array}{c}2\ell-1\\2\ell-2k+1\end{array}\!\!\right)\right\}\qquad (\ell\geq2),
\end{equation}
with $\alpha_1=-\frac{1}{4}$ and $\beta_1=\frac{5}{8}$. Here $\mu_j$
are given by the recurrence relation \eqref{expansion-Wn-muj}.
\end{theorem}

\begin{proof}
 Let
\begin{equation*}\label{expansion-Wn-alphaj-betaj}
W_n\sim\frac{\pi}{2}\left(1+\sum_{\ell=1}^{\infty}\frac{\alpha_{\ell}}{(n+\beta_{\ell})^{2\ell-1}}\right) \qquad (n\to\infty),
\end{equation*}
where $\alpha_\ell$ and $\beta_\ell$ are real numbers to be determined. This can be written as
\begin{equation}\label{expansion-Wn-alphaj-betajre}
\frac{2}{\pi}W_n\sim 1+\sum_{j=1}^{\infty}\frac{\alpha_j}{n^{2j-1}}\left(1+\frac{\beta_j}{n}\right)^{-2j+1}.
\end{equation}
Direct computation yields
\begin{align*}
\sum_{j=1}^{\infty}\frac{\alpha_j}{n^{2j-1}}\left(1+\frac{\beta_j}{n}\right)^{-2j+1}
&\sim\sum_{j=1}^{\infty}\frac{\alpha_j}{n^{2j-1}}\sum_{k=0}^{\infty}\left(\!\!\begin{array}{c}-2j+1\\k\end{array}\!\!\right)\frac{\beta_j^k}{n^k}\\
&\sim
\sum_{j=1}^{\infty}\frac{\alpha_j}{n^{2j-1}}\sum_{k=0}^{\infty}(-1)^{k}\left(\!\!\begin{array}{c}k+2j-2\\k\end{array}\!\!\right)\frac{\beta_j^k}{n^k}\\
&\sim
\sum_{j=1}^{\infty}\sum_{k=0}^{j-1}\alpha_{k+1}\beta_{k+1}^{j-k-1}(-1)^{j-k-1}\left(\!\!\begin{array}{c}j+k-1\\j-k-1\end{array}\!\!\right)\frac{1}{n^{j+k}},
\end{align*}
which can be written as
\begin{equation}\label{expansion-Wn-written}
\sum_{j=1}^{\infty}\frac{\alpha_j}{n^{2j-1}}\left(1+\frac{\beta_j}{n}\right)^{-2j+1}
\sim\sum_{j=1}^{\infty}\left\{\sum_{k=1}^{\lfloor
\frac{j+1}{2}\rfloor}\alpha_{k}\beta_{k}^{j-2k+1}(-1)^{j-1}\left(\!\!\begin{array}{c}j-1\\j-2k+1\end{array}\!\!\right)\right\}\frac{1}{n^j}.
\end{equation}
It then follows from \eqref{expansion-Wn-alphaj-betajre} and \eqref{expansion-Wn-written} that
\begin{equation}\label{expansion-Wn-writtenhave}
\frac{2}{\pi}W_n\sim1+
\sum_{j=1}^{\infty}\left\{\sum_{k=1}^{\lfloor
\frac{j+1}{2}\rfloor}\alpha_{k}\beta_{k}^{j-2k+1}(-1)^{j-1}\left(\!\!\begin{array}{c}j-1\\j-2k+1\end{array}\!\!\right)\right\}\frac{1}{n^j}.
\end{equation}

On the other hand, we have from \eqref{asymptotic-Wn-fromLemma-exp} that
\begin{equation}\label{asymptotic-Wn-fromLemma-expre}
\frac{2}{\pi}W_n\sim 1+\sum_{j=1}^{\infty}\frac{\mu_j}{n^j}.
\end{equation}
Equating coefficients of $n^{-j}$ on the right-hand sides of \eqref{expansion-Wn-writtenhave} and \eqref{asymptotic-Wn-fromLemma-expre}, we obtain
\begin{equation}\label{expansion-Wn-Equating-coefficients-muj}
\mu_j=\sum_{k=1}^{\lfloor
\frac{j+1}{2}\rfloor}\alpha_{k}\beta_{k}^{j-2k+1}(-1)^{j-1}\left(\!\!\begin{array}{c}j-1\\j-2k+1\end{array}\!\!\right)\qquad (j\in\mathbb{N}).
\end{equation}
Setting $j=2\ell-1$ and $j=2\ell$  in
\eqref{expansion-Wn-Equating-coefficients-muj}, respectively,
we find
\begin{equation}\label{expansion-Wn-Equating-coefficients-2ell-1}
\mu_{2\ell-1}=\sum_{k=1}^{\ell}\alpha_{k}\beta_{k}^{2\ell-2k}\left(\!\!\begin{array}{c}2\ell-2\\2\ell-2k\end{array}\!\!\right)
\end{equation}
and
\begin{equation}\label{expansion-Wn-Equating-coefficients-muj-2ell}
\mu_{2\ell}
=-\sum_{k=1}^{\ell}\alpha_{k}\beta_{k}^{2\ell-2k+1}\left(\!\!\begin{array}{c}2\ell-1\\2\ell-2k+1\end{array}\!\!\right).
\end{equation}

For $\ell=1$, we obtain from
\eqref{expansion-Wn-Equating-coefficients-2ell-1} and
\eqref{expansion-Wn-Equating-coefficients-muj-2ell}
\[\alpha_1=\mu_1=-\frac{1}{4}\quad\mbox{and}\quad \beta_1=-\frac{\mu_2}{\alpha_1}=\frac{5}{8},\]
and for $\ell\geq2$ we have
\begin{equation*}
\mu_{2\ell-1}=\sum_{k=1}^{\ell-1}\alpha_{k}\beta_{k}^{2\ell-2k}\left(\!\!\begin{array}{c}2\ell-2\\2\ell-2k\end{array}\!\!\right)+\alpha_{\ell}
\end{equation*}
and
\begin{equation*}
\mu_{2\ell}=-\sum_{k=1}^{\ell-1}\alpha_{k}\beta_{k}^{2\ell-2k+1}\left(\!\!\begin{array}{c}2\ell-1\\2\ell-2k+1\end{array}\!\!\right)-(2\ell-1)\alpha_{\ell}\beta_{\ell}.
\end{equation*}
We then obtain the recurrence relations
\eqref{Chen-expansion-WnCoefficients-alphaj} and
\eqref{Chen-expansion-WnCoefficients-betaj}.
The proof of Theorem \ref{Thm3-Expansion-Wn} is complete.
\end{proof}
\vspace{0.2cm}

We now give explicit numerical values
of the first few $\alpha_\ell$ and $\beta_\ell$ by using the recurrence relations
\eqref{Chen-expansion-WnCoefficients-alphaj} and
\eqref{Chen-expansion-WnCoefficients-betaj}. This demonstrates the ease with which
the constants $\alpha_\ell$ and $\beta_\ell$ in
\eqref{expansion-Wn-alphaj-betaj} can be determined. We find
\begin{align*}
&\alpha_1=-\frac{1}{4},\quad \beta_1=\frac{5}{8},\\
&\alpha_2=\mu_3-\alpha_1\beta_1^2=-\frac{11}{128}-\left(-\frac{1}{4}\right)\cdot\left(\frac{5}{8}\right)^2=\frac{3}{256},\\
&\beta_2=-\frac{\mu_4+\alpha_1\beta_1^3}{3\alpha_2}=-\frac{\frac{83}{2048}+\left(-\frac{1}{4}\right)\cdot\left(\frac{5}{8}\right)^3}{3\cdot\left(\frac{3}{256}\right)}=\frac{7}{12},\\
&\alpha_3=\mu_5-\alpha_1\beta_1^4-6\alpha_2\beta_2^2=-\frac{143}{8192}-\left(-\frac{1}{4}\right)\cdot\left(\frac{5}{8}\right)^4-6\cdot\left(\frac{3}{256}\right)\cdot\left(\frac{7}{12}\right)^2=-\frac{53}{16384},\\
&\beta_3=-\frac{\mu_6+\alpha_1\beta_1^5+10\alpha_2\beta_2^3}{5\alpha_3}=-\frac{\frac{625}{65536}+\left(-\frac{1}{4}\right)\cdot\left(\frac{5}{8}\right)^5+10\cdot\left(\frac{3}{256}\right)\cdot\left(\frac{7}{12}\right)^3}{5\cdot\left(-\frac{53}{16384}\right)}=\frac{2113}{3816}.
\end{align*}
Continuation of this procedure then enables the following coefficients to be derived:
\begin{eqnarray*}
&&\alpha_4=\frac{224573}{93782016},\quad
\beta_4=\frac{22119189899}{41134587264},\\
&&\alpha_5=-\frac{596297240983745796931}{176651089583152098705408},\quad \beta_5=\frac{38909478384301921254232134966821}{73585322683584986068354328660352}.
\end{eqnarray*}
We then obtain the following explicit asymptotic expansion:
\begin{align}
W_n&\sim\frac{\pi}{2}\bigg(1-\frac{\frac{1}{4}}{n+\frac{5}{8}}+\frac{\frac{3}{256}}{(n+\frac{7}{12})^3}-\frac{\frac{53}{16384}}{(n+\frac{2113}{3816})^5}\nonumber\\
&\qquad\quad+\frac{\frac{224573}{93782016}}{(n+\frac{22119189899}{41134587264})^7}-\frac{\frac{596297240983745796931}{176651089583152098705408}}{(n+\frac{38909478384301921254232134966821}{73585322683584986068354328660352})^9}+\cdots\bigg).\label{asymptotic-WnThm4}
\end{align}
Thus, we would appear to
obtain an alternating odd-type asymptotic expansion for $W_n$.
From a computational viewpoint,
(\ref{asymptotic-WnThm4}) is an improvement on the  formulas \eqref{asymptotic-Wnnamely} and
 \eqref{Wn-expansion-Elezovic}.

\begin{theorem}\label{Thm4-Expansion-Wn-exp}
The Wallis sequence has the following asymptotic expansion:
\begin{equation}\label{asymptotic-Wn-exp}
W_n\sim\frac{\pi}{2}\exp\bigg(\sum_{\ell=1}^{\infty}\frac{\omega_{\ell}}{(n+\frac{1}{2})^{2\ell-1}}\bigg)\qquad (n\to\infty)
\end{equation}
with the coefficients $\omega_\ell$  given by the recurrence relation
\begin{equation}\label{asymptotic-Wn-exp-coefficients-omega}
\omega_1=-\frac{1}{4}\quad\mbox{and}\quad \omega_{\ell}=\nu_{2\ell-1}-\sum_{k=1}^{\ell-1}\omega_{k}\left(\frac{1}{2}\right)^{2\ell-2k}\left(\!\!\begin{array}{c}2\ell-2\\2\ell-2k\end{array}\!\!\right)\qquad (\ell\geq2),
\end{equation}
where the $\nu_j$ are given in \eqref{asymptotic-Wn-coefficient-nuj}.
\end{theorem}

\begin{proof}
 Let
\begin{equation*}
W_n\sim\frac{\pi}{2}\exp\bigg(\sum_{\ell=1}^{\infty}\frac{\omega_{\ell}}{(n+\frac{1}{2})^{2\ell-1}}\bigg)\qquad (n\to\infty),
\end{equation*}
where $\omega_\ell$  are real numbers to be determined. This can be written as
\begin{equation*}
\ln\left(\frac{2}{\pi}W_n\right)\sim \sum_{j=1}^{\infty}\frac{\omega_j}{n^{2j-1}}\left(1+\frac{1}{2n}\right)^{-2j+1}.
\end{equation*}
The choice $\beta_j=\frac{1}{2}$ in \eqref{expansion-Wn-written}, with $\alpha_j$  replaced by $\omega_j$, yields
\begin{equation*}
\sum_{j=1}^{\infty}\frac{\omega_j}{n^{2j-1}}\left(1+\frac{1}{2n}\right)^{-2j+1}
\sim\sum_{j=1}^{\infty}\left\{\sum_{k=1}^{\lfloor
\frac{j+1}{2}\rfloor}\omega_{k}\left(\frac{1}{2}\right)^{j-2k+1}(-1)^{j-1}\left(\!\!\begin{array}{c}j-1\\j-2k+1\end{array}\!\!\right)\right\}\frac{1}{n^j}.
\end{equation*}
We then obtain
\begin{equation}\label{expansion-Wn-obtain}
\ln\left(\frac{2}{\pi}W_n\right)\sim \sum_{j=1}^{\infty}\left\{\sum_{k=1}^{\lfloor
\frac{j+1}{2}\rfloor}\omega_{k}\left(\frac{1}{2}\right)^{j-2k+1}(-1)^{j-1}\left(\!\!\begin{array}{c}j-1\\j-2k+1\end{array}\!\!\right)\right\}\frac{1}{n^j}.
\end{equation}

On the other hand, we have from \eqref{asymptotic-Wn} that
\begin{equation}\label{asymptotic-Wnre}
\ln\left(\frac{2}{\pi}W_n\right)\sim\sum_{j=1}^{\infty}\frac{\nu_j}{n^j}.
\end{equation}
Equating coefficients of $n^{-j}$ on the right-hand sides of
\eqref{expansion-Wn-obtain} and \eqref{asymptotic-Wnre}, we obtain
\begin{equation}\label{expansion-Wn-Equating-coefficients}
\nu_j=\sum_{k=1}^{\lfloor
\frac{j+1}{2}\rfloor}\omega_{k}\left(\frac{1}{2}\right)^{j-2k+1}(-1)^{j-1}\left(\!\!\begin{array}{c}j-1\\j-2k+1\end{array}\!\!\right)\qquad (j\in\mathbb{N}).
\end{equation}
Setting $j=2\ell-1$  in \eqref{expansion-Wn-Equating-coefficients}, we find
\begin{equation}\label{Chen-expansion-Gamma-Equating-coefficients-2ell-1}
\nu_{2\ell-1}=\sum_{k=1}^{\ell}\omega_{k}\left(\frac{1}{2}\right)^{2\ell-2k}\left(\!\!\begin{array}{c}2\ell-2\\2\ell-2k\end{array}\!\!\right).
\end{equation}
Substitution of $\ell=1$ in
\eqref{Chen-expansion-Gamma-Equating-coefficients-2ell-1} yields $\omega_1=\nu_1=-\frac{1}{4}$, and for $\ell\geq2$ we have
\begin{equation*}
\nu_{2\ell-1}=\sum_{k=1}^{\ell-1}\omega_{k}\left(\frac{1}{2}\right)^{2\ell-2k}\left(\!\!\begin{array}{c}2\ell-2\\2\ell-2k\end{array}\!\!\right)+\omega_{\ell}.
\end{equation*}
We then obtain the recurrence relation \eqref{asymptotic-Wn-exp-coefficients-omega}. The proof of Theorem \ref{Thm4-Expansion-Wn-exp} is complete.
\end{proof}

\begin{remark}
Setting $j=2\ell$  in \eqref{expansion-Wn-Equating-coefficients}, we find
\begin{equation}\label{expansion-Wn-Equating-coefficients-2ell}
\nu_{2\ell}
=-\sum_{k=1}^{\ell}\omega_{k}\left(\frac{1}{2}\right)^{2\ell-2k+1}\left(\!\!\begin{array}{c}2\ell-1\\2\ell-2k+1\end{array}\!\!\right).
\end{equation}
For $\ell=1$ in
\eqref{expansion-Wn-Equating-coefficients-2ell} this yields $\omega_1=-2\nu_2=-\frac{1}{4}$, and for $\ell\geq2$ we have
\begin{equation*}
\nu_{2\ell}=-\sum_{k=1}^{\ell-1}\omega_{k}\left(\frac{1}{2}\right)^{2\ell-2k+1}\left(\!\!\begin{array}{c}2\ell-1\\2\ell-2k+1\end{array}\!\!\right)-(\ell-\frac{1}{2})\,\omega_{\ell}.
\end{equation*}
We then obtain the alternative recurrence relation for the coefficients $\omega_j$ in \eqref{asymptotic-Wn-exp} in terms of the even coefficients $\nu_j$:
\begin{equation}
\omega_1=-\frac{1}{4}\quad\mbox{and}\quad\omega_{\ell}=-\frac{2}{2\ell-1}\left\{\nu_{2\ell}+\sum_{k=1}^{\ell-1}\omega_{k}\left(\frac{1}{2}\right)^{2\ell-2k+1}\left(\!\!\begin{array}{c}2\ell-1\\2\ell-2k+1\end{array}\!\!\right)\right\}\qquad (\ell\geq2).
\end{equation}
\end{remark}
Hence, from \eqref{asymptotic-Wn-exp},  we obtain the following explicit asymptotic expansion:
\begin{equation}\label{asymptotic-Wn-expnamely}
W_n\sim\frac{\pi}{2}\exp\bigg(-\frac{\frac{1}{4}}{n+\frac{1}{2}}+\frac{\frac{1}{96}}{(n+\frac{1}{2})^3}-\frac{\frac{1}{320}}{(n+\frac{1}{2})^5}+\frac{\frac{17}{7168}}{(n+\frac{1}{2})^7}-\frac{\frac{31}{9216}}{(n+\frac{1}{2})^9}+\cdots\bigg).
\end{equation}
This would appear to be an alternating odd-type expansion for $W_n$.
From a computational viewpoint,
\eqref{asymptotic-Wn-expnamely} is an improvement on the  formulas \eqref{Thm-p=0-q--1/4} and
 \eqref{asymptotic-Wnnamely-exp}.


\end{CJK*}

\bigskip

\medskip

\bigskip
\medskip

\end{document}